\documentclass[11pt,a4paper]{amsart}    

\usepackage{amssymb,amsthm,amsmath}
\usepackage{graphicx}
\usepackage{amssymb,amsmath}
\usepackage[left=3.2cm,top=3.5cm,right=3.2cm,bottom=3cm]{geometry}
\usepackage{bbm}
\usepackage{graphicx,color,pgf}

\newcommand{\R}{\mathbb{R}}
\renewcommand{\S}{\mathbb{S}}
\newcommand{\C}{\mathbb{C}}
\newcommand{\M}{\mathcal{M}}

\newcommand{\eps}{\varepsilon}
\newcommand{\inj}{\text{inj}}
\newcommand{\ra}{\rightarrow}
\newcommand{\vol}{\text{vol}}

\newcommand{\vphi}{\varphi}

\newcommand{\diam}{\text{diam}}
\newcommand{\dvol}{\text{dvol}}

\newcommand{\Bal}{\text{Bal}}
\newcommand{\dsep}{\text{dsep}}
\renewcommand{\t}{\mathrm{t}}

\newcommand{\K}{\mathcal{K}}

\newcommand{\Kk}{\mathbb{K}}

\def\Xint#1{\mathchoice
{\XXint\displaystyle\textstyle{#1}}%
{\XXint\textstyle\scriptstyle{#1}}%
{\XXint\scriptstyle\scriptscriptstyle{#1}}%
{\XXint\scriptscriptstyle\scriptscriptstyle{#1}}%
\!\int}
\def\XXint#1#2#3{{\setbox0=\hbox{$#1{#2#3}{\int}$}
\vcenter{\hbox{$#2#3$}}\kern-.5\wd0}}
\def\dashint{\Xint-}

\newtheorem{theorem}{Theorem}[section]
\newtheorem{lemma}[theorem]{Lemma}
\newtheorem{proposition}[theorem]{Proposition}
\newtheorem{corollary}[theorem]{Corollary}
\theoremstyle{definition}
\newtheorem{definition}[theorem]{Definition}

\theoremstyle{remark}
\newtheorem{remark}[theorem]{Remark}

\begin{document}

\title[Separation distance of minimal Green energy points]{On the separation distance of minimal Green energy points on compact Riemannian manifolds}

\thanks{This research was partially supported by Ministerio de Econom\'ia y Competitividad under grants MTM2017-83816-P, MTM2014-57590-P and MTM2013-46337-C2-1-P, Banco Santander and University of Cantabria under grant 21.SI01.64658, and a PhD grant by the University of Cantabria. The author is a postdoctoral fellow funded by the FWO  Flanders project EOS  30889451.}

\author{Juan G. Criado del Rey}

\address{Juan G. Criado del Rey. Departement Wiskunde, KU Leuven.}
\email{juan.gcriadodelrey@kuleuven.be}


\maketitle

\begin{abstract}
In this article we study point configurations minimizing the discrete energy on a compact Riemannian manifold, where the energy kernel is taken to be the Green's function for the Laplacian. We show that every point in a minimizing configuration lies inside an open set called harmonic ball where no other point can enter, and that the minimum distance between any two distinct points has the optimal asymptotic order. We compute explicit bounds for the minimum distance in the case of Compact Rank One Symmetric Spaces.
\end{abstract}

\section{Introduction and main results}

In this article we continue the study initiated in \cite{CriadoDelReyDiscreteAndContinuous} about the role of the Green's function for the Laplacian on a compact Riemannian manifold $(\M,g)$ of dimension $n \geq 2$ as a tool for obtaining well--distributed points. The Green's function $G(x,y)$ is symmetric, smooth off the diagonal, and it satisfies
\begin{equation}\label{eq:def_green_function}
\Delta G(x,\cdot) = \delta_x-V^{-1}\vol
\end{equation} in the sense of distributions, where $\vol$ is the Riemannian volume form or density, and $V = \vol(\M)$ is the volume of $\M$ (see \cite[Ch. 4]{Aubin}). Throughout this article we adopt the sign convention for the Laplacian so that its representation in local coordinates becomes
\[
\Delta = -\frac{1}{\sqrt{|g|}}\partial_i \sqrt{|g|}g^{ij}\partial_j,
\] where $|g|$ is the determinant of the metric, $g^{ij}$ are the coefficients of $g^{-1}$, and we use the Einstein summation convention. It is customary to fix $G$ from \eqref{eq:def_green_function} by imposing $\int G = 0$. The \emph{Green energy} of a set of $N\geq 2$ distinct points $\omega_N = \{x_1, ..., x_N\}\subset\M$ is the sum of the pairwise interactions between the points. This is,
\begin{equation}\label{eq:def_green_energy}
E_G(\omega_N) = E_G(x_1, ..., x_N) = \sum_{i\neq j}G(x_i,x_j).
\end{equation}

When $\M = \S^2$ is the two dimensional sphere, $G(x,y)$ is (up to scaling by positive constants) the logarithmic kernel $\log\|x-y\|^{-1}$ and thus $E_G$ corresponds to the classical logarithmic energy given by
\begin{equation}\label{eq:def_elog}
E_{\log}(\omega_N) = \sum_{i\neq j}\log\|x_i-x_j\|^{-1}.
\end{equation} When $\M = \S^n$ for $n > 2$, the Green's function is proportional to $\|x-y\|^{2-n}$ plus higher order terms (see \cite[A. 1]{CriadoDelReyDiscreteAndContinuous}), so in this case one would expect $E_G$ to share some properties with the Riesz $s$--energy
\begin{equation}\label{eq:def_riesz}
E_s(\omega_N) = \sum_{i\neq j}\|x_i-x_j\|^{-s}
\end{equation} for the value of the parameter $s = n-2$.

Points minimizing $E_{\log}$ and $E_s$ on the unit sphere tend to exhibit good distribution properties. Classical Potential Theory \cite{Landkof} shows, for example, that point configurations minimizing $E_{\log}$ or $E_s$ for $s\in (0,n)$ are asymptotically uniformly distributed. This means that if $(\omega_N^*)_N$ is a sequence of minimal energy points, then the empirical measures $\frac{1}{N}\sum_{x\in\omega_N^*}\delta_x$ converge weak$^*$ to the uniform probability measure on $\S^n$. In particular, minimal logarithmic and Riesz energy points are feasible candidates for their use in quasi Monte Carlo integration on $\S^n$.

The \emph{separation distance} $\dsep(\omega_N)$ of a point configuration $\omega_N$ on a subset of Euclidean space is the least distance between any two pair of distinct points. This is,
\begin{equation}\label{eq:def_dsep}
\dsep(\omega_N) = \min\{\|x-y\| : x,y\in\omega_N,\,x\neq y\}.
\end{equation} Rakhmanov, Saff and Zhou proved in \cite{Saff94} that minimal logarithmic energy points are never too close one to each other, in the sense that there is a positive constant $c$ such that, if $(\omega_N^*)_N$ is a sequence of minimizing configurations for $E_{\log}$ on $\S^2$, then
\begin{equation}\label{eq:dsep_log}
\dsep(\omega_N) \geq cN^{-1/2}.
\end{equation} This result was later improved by Dubickas \cite{Dubickas96} and Dragnev \cite{DragnevSeparation2002}. It is easy to see that the asymptotic order in \eqref{eq:dsep_log} cannot be improved. More generally, the asymptotic order for the separation distance on any $n$--dimensional submanifold of $\R^m$ is at most $\Omega(N^{-1/n})$. For this reason, we say that a sequence of point configurations $(\omega_N)_N$ is \emph{well--separated} if there is a positive constant $c$ such that $\dsep(\omega_N) \geq cN^{-1/n}$.

To our knowledge, it is still unknown whether minimal logarithmic energy points are well--separated on $\S^n$ for $n > 2$. Damelin and Maymeskul proved \cite{Damelin} that sequences of minimal Riesz $s$--energy points on $\S^n$ for $n > 2$ satisfy
\[
\dsep(\omega_N) \geq cN^{-\frac{1}{s+2}}
\] for $s\in (0,n-2]$, which shows that minimal Riesz energy points are well--separated when $s = n-2$. As $s$ increases, the problem becomes more similar to the best packing problem (because the least distance in \eqref{eq:def_riesz} dominates the total energy), so that minimal Riesz energy points on $\S^n$ are well--separated when $s\geq n-2$ (see \cite{RieszSphericalPotentials} for $s\in (n-2,n-1)$, \cite{Gotz} for $s = n-1$, \cite{KuijlaarsSeparation} for $s\in(n-1,n)$ and \cite{KuijlaarsAsymptotics} for $s\in (n,\infty)$). A general result due to Dahlberg \cite{Dahlberg} states that minimal Riesz energy points with $s = n-1$ are well--separated on $C^{1,\alpha}$ hypersurfaces of $\R^{n+1}$. In \cite{HarmonicEnsembleCarlos} the authors study a particular determinantal point process on the sphere called harmonic ensemble. Points drawn from the harmonic ensemble have low expected Riesz and logarithmic energy and, moreover, they are likely to have a separation distance of order $\Omega(N^{-\frac{2n+2}{n^2+2n}})$.

In \cite{AxisSupported,Hypersphere,SloanBirthday} the authors study minimal energy configurations on the sphere in the presence of an external field $Q:\S^n\ra\R\cup\{+\infty\}$ and they develop a technique to find bounds on the separation distance of (free--field) minimal energy points. In this article we adopt a similar approach, but using a slightly different language. The main reason why their technique cannot be directly applied to general Riemannian manifolds has to do with the mean value property (MVP) for harmonic functions. When $\M = \S^n$, if we pick a geodesic ball $B$ and a harmonic function $\vphi:B\ra\R$, the MVP states that
\begin{equation}\label{eq:mvp_geod}
\dashint_B\vphi = \vphi(p),
\end{equation} where $\dashint_B = \vol(B)^{-1}\int$ is the average value of the function. This property does not hold for a general $\M$ without further assumptions on the symmetries of $\M$. There exist, however, certain open subsets of $\M$ called harmonic balls (see Definition \ref{def:harmonic_ball}) for which the MVP holds true. These sets will be the natural ``areas of influence'' of points belonging to a minimizing configuration for the Green energy. Harmonic balls were introduced in \cite{SjodinHarmonic} for the case $\M = \R^n$, and they were studied in \cite{Varchenko,HedenmalmHyperbolic,RossHeleShaw,RossMonge} in the context of Hele--Shaw flow problems on manifolds. The authors in \cite{GustafssonPartialBalayage} study harmonic balls within the theory of partial balayage on Riemannian manifolds. The rigorous definition of harmonic balls will be given at the beginning of Section \ref{sec:harmonic_balls}. By now, let us denote the harmonic ball with center $p$ and radius $a$ by $B^h(p,a)$, and keep in mind that it is an open set whose size increases with $a$. Our first result states that every point in a minimizing configuration for the Green energy is surrounded by a harmonic ball (of suitable radius) where no other point in the configuration can enter. The proof of this theorem can be found in Section \ref{sec:harmonic_balls}.

\begin{theorem}\label{thm:main1} Let $\M$ be a compact Riemannian manifold of dimension $n\geq 2$, and let $\omega_N^* = \{x_1^*, ..., x_N^*\}$ be a configuration of $N \geq 2$ distinct points minimizing the discrete Green energy on $\M$. Then, for every $i\neq j$,
\[
x_i^*\notin B^h\left(x_j^*,\frac{V}{N-1}\right),
\] where $V = \textup{vol}(\M)$ is the volume of $\M$.
\end{theorem}

When the manifold $\M$ is symmetric enough, it is possible to obtain a simple radially symmetric expression for the Green's function. As a consequence, in this case harmonic balls are just geodesic balls. Gustafsson and Roos proved \cite[Theorem 14]{GustafssonPartialBalayage} that in dimension $2$, and for a sufficiently small radius, geodesic and harmonic balls coincide as families if and only if $\M$ is a surface with constant Gaussian curvature. Here we will show that these concepts also coincide for a family of very symmetric manifolds called locally harmonic Blaschke manifolds (see Theorem \ref{thm:harmonic_geodesic}). In particular, the statement holds true for the Compact Rank One Symmetric Spaces (CROSS), consisting of the sphere $\S^n$, the projective spaces $\mathbb{R}\mathbb{P}^m$, $\mathbb{C}\mathbb{P}^m$, $\mathbb{H}\mathbb{P}^m$, and the Cayley plane $\mathbb{O}\mathbb{P}^2$ (we adopt the convention that the projective spaces and the Cayley plane all have diameter $\pi/2$). With this result in hand, the following separation result, whose proof can be found in Section \ref{sec:cross}, is an easy consequence of Theorem \ref{thm:main1}. If $\omega_N = \{x_1, ..., x_N\}$ is a set of $N\geq 2$ points on a Riemannian manifold, we define
\[
\dsep(\omega_N) = \min\{d(x,y) : x,y\in \omega_N, x\neq y\},
\] where $d(x,y)$ is the intrinsic Riemannian distance between $x$ and $y$.

\begin{theorem}\label{thm:main2} Let $\M$ be a CROSS, and let $(\omega_N^*)_{N\geq 2}$, be a sequence of minimal Green energy points. Then, for every $N\geq 2$,
\[
\textup{dsep}(\omega_N^*) \geq C_\M(N-1)^{-1/n},
\] where the values of the constants $C_\M$ are given by

\begin{center}
\begin{tabular}{|c||c|c|c|c|c|}
\hline
$\M$ & $\S^n$ & $\mathbb{R}\mathbb{P}^m$ & $\C\mathbb{P}^m$ & $\mathbb{H}\mathbb{P}^m$ & $\mathbb{O}\mathbb{P}^2$\\
\hline
$C_\M$ & $\left(\frac{n\sqrt{\pi}\,\Gamma\left(\frac{n}{2}\right)}{\Gamma\left(\frac{n+1}{2}\right)}\right)^{1/n}$ & $\left(\frac{m\sqrt{\pi}\,\Gamma\left(\frac{m}{2}\right)}{2\,\Gamma\left(\frac{m+1}{2}\right)}\right)^{1/m}$ & $1$ & $\left(\frac{1}{2m+1}\right)^{1/4m}$ & $\left(\frac{1}{165}\right)^{1/16}$\\
\hline
\end{tabular}
\end{center}
\end{theorem}

\begin{remark} We do not know whether the values for the constants given in Theorem \ref{thm:main2} are sharp or not. However, when $\M = \S^2$, we get $C_{\S^2} = 2$. Thus $\dsep(\omega_N^*) \geq \frac{2}{(N-1)^{1/2}}$. Since $\|x-y\|$ and $d(x,y)$ are approximately the same for points $x$ and $y$ very close one to each other, this bound is analogous to the one obtained in \cite[Theorem 2]{DragnevSeparation2002} for minimal logarithmic energy points on $\S^2$.
\end{remark}

In the general case we will adapt the results and arguments in \cite{BlankDivergenceForm} and \cite{BlankManifolds} to show that every harmonic ball contains a geodesic ball of proportional radius. Combining this fact with Theorem \ref{thm:main1} yields the following theorem:

\begin{theorem}\label{thm:main3} Let $\M$ be a compact Riemannian manifold of dimension $n\geq 2$. There exists a constant $c>0$ such that, for every configuration $\omega_N^* = \{x_1^*, ..., x_N^*\}$ of $N\geq 2$ distinct points on $\M$ minimizing the discrete Green energy,
\[
\textup{dsep}(\omega_N^*) \geq cN^{-1/n}.
\] This is, minimal Green energy points are well--separated.
\end{theorem}

The proof of Theorem \ref{thm:main3} will be presented in Section \ref{sec:general}.

\section{Harmonic balls and the areas of influence of minimal Green energy points}\label{sec:harmonic_balls}

This section will be devoted to the definition and basic properties of harmonic balls, and the proof of Theorem \ref{thm:main1}. In \cite{GustafssonPartialBalayage} Gustafsson and Roos establish the foundations of the theory of partial balayage on Riemannian manifolds. In its most simple formulation, the partial balayage $\nu = \Bal(\sigma,\lambda)$ of a signed measure $\sigma$ with respect to a reference signed measure $\lambda$ is another signed measure obtained by rearranging the mass of $\sigma$ in such a way that $\nu$ stays below $\lambda$, using the least amount of energy (see \cite[Definition 1]{GustafssonPartialBalayage}). This formulation, however, is not enough if we want to define the balayage of measures with infinite energy such as the delta distribution $\sigma = \delta_p$.

In this article we will not make use of the definition of partial balayage in its full generality and we will consider only the case when $\sigma_-$ has finite energy (as in the first part of Section 5 in \cite{GustafssonPartialBalayage}). We begin by introducing the notion of harmonic ball.

\begin{remark} Throughout this article it will be convenient to use the notation
\[
\t_a = 1-aV^{-1},
\] where $V = \vol(\M)$ is the volume of $\M$, and $a$ is some real number with $0 < a < V$. Observe that $\t_a > 0$ and $\t_a > \t_b$ if $0 < a < b < V$. The number $\t_a$ equals minus the mass of the charge distribution $\sigma = a\delta_p-\vol$ and the notation comes from \cite{GustafssonPartialBalayage}.
\end{remark}

\begin{definition}\label{def:harmonic_ball} Let $a$ be a real number with $0 < a < V$, where $V = \vol(\M)$ is the volume of $\M$, and let us denote by $G^{\delta_p}(x) = G(p,x)$ the Green's function centered at $p$. Consider the set
\[
\mathcal{K}_a = \{u\in W^{1,2}(\M) : u \leq aG^{\delta_p}\},
\] where $W^{1,2}(\M)$ denotes the Sobolev space of square integrable functions on $\M$ whose first weak derivatives are also square integrable. The functional
\[
J_a(u) = \int_{\M}\langle\nabla u,\nabla u\rangle-2\t_a u,
\] has a unique minimizer $u_a$ in $\K_a$. The minimizer $u_a$ is a continuous function, and in terms of it we define the \emph{harmonic ball} with center $p$ and volume $a$ as the non--contact set
\[
B^h(p,a) = \{x\in\M : u_a(x) < aG^{\delta_p}(x)\}.
\]
\end{definition}

\begin{remark}\label{rmk:harmonic_ball} Since $u_a$ is continuous (see \cite[Remark 3]{GustafssonPartialBalayage} and also \cite[Lemma 3]{GardinerSjodin2007}) the harmonic ball $B^h(p,a)$ is an open set and, moreover, since $\lim_{x\ra p}aG^{\delta_p}(x) = +\infty$, we always have $p\in B^h(p,a)$.
\end{remark}

\begin{remark}\label{rmk:frostman} Theorem 11 in \cite{GustafssonPartialBalayage} shows that the free boundary $\partial B^h(p,a)$ has (Riemannian) measure zero and that
\begin{equation}\label{eq:lap_ua_prev}
\Delta u_a = \t_a\vol\chi_{\{u_a < aG^{\delta_p}\}}+(a\delta_p-aV^{-1}\vol)\chi_{\{u_a = aG^{\delta_p}\}}
\end{equation} in the sense of distributions, where $\chi_B$ denotes the indicator function of a subset $B\subseteq\M$. Moreover, since $p\in B^h(p,a) = \{u_a < aG^{\delta_p}\}$ (see Remark \ref{rmk:harmonic_ball}), \eqref{eq:lap_ua_prev} becomes simply
\begin{equation}\label{eq:lap_ua}
\Delta u_a = \t_a\vol\chi_{\{u_a < aG^{\delta_p}\}}-aV^{1}\vol\chi_{\{u_a = aG^{\delta_p}\}}.
\end{equation} It follows that
\[
\Delta(u_a-aG^{\delta_p}) = \vol-a\delta_p \qquad \text{ on }B^h(p,a)
\] in the sense of distributions. Hence, if $\vphi$ is harmonic on $B^h(p,a)$, then
\begin{equation}\label{eq:mvp_prev}
0 = \int_{B^h(p,a)}(u_a-aG^{\delta_p})\Delta\vphi = \left(\int_{B^h(p,a)}\vphi\right)-a\vphi(p).
\end{equation} Taking $\vphi \equiv 1$ yields that $\vol(B^h(p,a)) = a$, and thus \eqref{eq:mvp_prev} can be expressed in the more familiar MVP form
\begin{equation}\label{eq:mvp}
\dashint_{B^h(p,a)}\vphi = \vphi(p).
\end{equation}
\end{remark}

\begin{remark}\label{rmk:obstacle_ua} The function $u_a$ can be also seen as the solution to an obstacle--type problem. From the general definition of partial balayage \cite[Definition 5]{GustafssonPartialBalayage} and \cite[Theorem 7]{GustafssonPartialBalayage}, $u_a$ is the largest function $u$ such that, in the sense of distributions, satisfies
\begin{equation}\label{eq:obstacle_ua}
\left\{\begin{array}{l}
u \leq aG^{\delta_p},\\
\Delta u \leq \t_a.
\end{array}\right.
\end{equation}
\end{remark}

We now prove that harmonic balls are connected sets. The argument is the same as in \cite[Proposition 2.6]{HedenmalmHyperbolic}, where the authors consider harmonic balls on hyperbolic surfaces, but we include it here for completeness.

\begin{proposition}\label{prop:bharm_connected} $B^h(p,a)$ is a connected subset of $\M$.
\end{proposition}

\begin{proof} Let $B$ be the connected component of $B^h(p,a)$ containing $p$ (see Remark \ref{rmk:harmonic_ball}). Assume, to the contrary, that there is another connected component $B'$ of $B^h(p,a)$. Since $B$ is open and $p\in B$, the point $p$ is positive distance away from $B'$, so $\Delta aG^{\delta_p} = -aV^{-1}$ in $B'$. From \eqref{eq:lap_ua}, in $B'$ we have $\Delta u_a = 1-aV^{-1}$. Thus $\Delta u_a \geq \Delta aG^{\delta_p}$ in $B'$, and both functions coincide on $\partial B'$. By the Maximum Principle, $u_a \geq aG^{\delta_p}$ in $B'$, which is a contradiction.
\end{proof}

The next proposition shows that harmonic balls are nested according to their volume. A similar result is proved in \cite[Lemma 4.1]{BlankManifolds}, where the authors consider harmonic balls intersecting a reference submanifold with boundary. Again, the argument here is very similar, but we include it for completeness.

\begin{proposition}\label{prop:bharm_nested} Let $0 < a \leq b < V$. Then $B^h(p,a)\subseteq B^h(p,b)$.
\end{proposition}

\begin{proof} Consider the obstacle problems
\[
(P_a)\ \max u\ :\ \left\{\begin{array}{l}
u\leq aG^{\delta_p},\\
\Delta u \leq \t_a.
\end{array}\right. \qquad (P_b)\ \max u\ :\ \left\{\begin{array}{l}
u\leq bG^{\delta_p},\\
\Delta u \leq \t_b.
\end{array}\right.
\] From Remark \ref{rmk:obstacle_ua}, we know that $\Delta u_b \leq 1-bV^{-1}$, so
\begin{align*}
\Delta(u_b-bG^{\delta_p}+aG^{\delta_p}) &\leq (1-bV^{-1})\vol-b\delta_p+bV^{-1}\vol+a\delta_p-aV^{-1}\vol\\
&= (1-aV^{-1})\vol+(a-b)\delta_p,\\
&\leq (1-aV^{-1})\vol,
\end{align*} because $a \leq b$. Moreover, since $u_b-bG^{\delta_p} \leq 0$, we have that
\[
u_b-bG^{\delta_p}+aG^{\delta_p} \leq aG^{\delta_p}.
\] Hence $u_b-bG^{\delta_p}+aG^{\delta_p}$ is a competing function for the problem $(P_a)$. Since $u_a$ solves $(P_a)$, we conclude that
\[
u_b-bG^{\delta_p}+aG^{\delta_p} \leq u_a,
\] or
\[
aG^{\delta_p}-u_a \leq bG^{\delta_p}-u_b.
\] The proposition follows.
\end{proof}

We are now almost ready to prove Theorem \ref{thm:main1}. First, we will first prove the following simple but useful lemma.

\begin{lemma}\label{lem:lemita} Let $q$ be a point different from $p$. Assume that there is a function $f$, continuous in a neighborhood of $q$, such that
\begin{enumerate}
\item $f(q) = aG^{\delta_p}(q)$,
\item $f \leq aG^{\delta_p}$,
\item $\Delta f \leq \t_a$.
\end{enumerate} Then $q\notin B^h(p,a)$.
\end{lemma}

\begin{proof} Since $f \leq aG^{\delta_p}$ and $\Delta f \leq \t_a$, from Remark \ref{rmk:obstacle_ua} we know that $f \leq u_a$. Since both $f$ and $u_a$ are continuous in a neighborhood of $q$, the inequality holds pointwise around $q$. Therefore,
\[
aG^{\delta_p}(q) = f(q) \leq u_a(q)
\] and the lemma follows.
\end{proof}

\begin{proof}[Proof of Theorem \ref{thm:main1}] By symmetry, it is enough to show that $x^*_{N-1}\notin B^h\left(x_N^*,\frac{V}{N-1}\right)$. Consider the function
\begin{align*}
h(x) &= \frac{V}{N-1}\left(G^{\delta_{x_N^*}}(x) + \sum_{i=1}^{N-2}G(x_i^*,x)\right)\\
&= \frac{V}{N-1}\left(E_G(x_1^*, ..., x_{N-2}^*,x,x_N^*)-\sum_{i\neq j}^{N-2}G(x_i^*,x_j^*)- \sum_{i=1}^{N-2}G(x_i^*,x_N^*)\right)
\end{align*} Since $\{x_1^*, ..., x_N^*\}$ is a minimizing configuration for $E_G$, the point $x_{N-1}^*$ is a global minimum for $h$. Hence, if we set
\[
f(x) = h(x_{N-1}^*)-\frac{V}{N-1}\sum_{i=1}^{N-2}G(x_i^*,x),
\] then $f$ is continuous in a neighborhood of $x_{N-1}^*$ and
\[
\frac{V}{N-1}G^{\delta_{x_N^*}}(x) \geq f(x)
\] for every $x\in\M$, with equality if $x = x_{N-1}^*$. Moreover,
\[
\Delta f = -\frac{V}{N-1}\sum_{i=1}^{N-2}\Delta G^{\delta_{x_i^*}} \leq \frac{N-2}{N-1}\vol = \left(1-\frac{1}{N-1}\right)\vol.
\] The proof concludes by applying Lemma \ref{lem:lemita} with $a = \frac{V}{N-1}$, $p = x_N^*$  and $q = x_{N-1}^*$.
\end{proof}

\section{Separation distance for symmetric spaces}\label{sec:cross}

In this section we will prove Theorem \ref{thm:main2}. The MVP \eqref{eq:mvp} is not true in general if one replaces $B^h(p,a)$ by some geodesic ball, and thus harmonic balls and geodesic balls are not the same subsets of $\M$ in the general case. There are some manifolds, however, for which the MVP for geodesic balls holds true. These are called \emph{locally harmonic manifolds}. We say that $\M$ is locally harmonic at a point $p$ if every sufficiently small geodesic sphere around $p$ has constant mean curvature. In this article we will be using an alternative definition of local harmonicity involving volume density.

\begin{remark}\label{rmk:injectivity} Recall that if $x\in\M$ is any point, and if $r>0$ is sufficiently small, we can define normal coordinates $\exp_x:\mathbb{B}_r \ra B(x,r)$ on the geodesic ball around $x$, where $\mathbb{B}_r = \{v\in\R^n : \|v\| = r\}$ is the ball centered at the origin of $T_x\M$ (which we identify with $\R^n$) with radius $r$. Normal coordinates have the advantage that they preserve the metric in the radial directions, in the sense that for every $v\in \mathbb{B}_r$ we have that $d(x,\exp_xv) = \|v\|$. The maximal $r>0$ such that $\exp_x:\mathbb{B}_r\ra B(x,r)$ is a diffeomorphism is called the \emph{injectivity radius} of $x$, and it is denoted by $\inj(x)$. The \emph{injectivity radius} of $\M$ is the number
\[
\inj(\M) = \inf_{x\in\M}\inj(x).
\] If $\M$ is compact, then $\inj(\M)$ is always positive, so we can use normal coordinates on $B(x,r)$ for every $0 < r < \inj(\M)$ not depending on $x$.

Another important concept is that of cut locus. The \emph{cut locus} of a point $x\in\M$ is the set of points $y\in\M\setminus\{x\}$ such that $y$ is conjugate to $x$ (see, for example, \cite[Definition 4.3.1]{JostRiemannian}) or two different minimizing geodesics emanating from $x$ arrive at $y$. We denote the cut locus of $x$ by $\text{Cut}(x)$. It can be shown that $\text{Cut}(x)$ has measure zero and that $d(x,\text{Cut}(x)) = \inj(x)$.

All these facts, which will be relevant when we define locally harmonic and Blaschke manifolds, are covered for instance in \cite{PetersenRiemannian}.
\end{remark}

\begin{definition}\label{def:volume_density} Let $\M$ be a Riemannian manifold, and let $x\in\M$ be any point. The \emph{volume density} $\omega(x,y) = \omega_x(y)$ is a continuous function on $\M\times\M$ whose local expression in normal coordinates around $x$ is
\begin{equation}\label{eq:local_expression_vol_dens}
\omega_x(y) = \sqrt{|g|}(y),
\end{equation} where $|g|$ is the determinant of the metric (see \cite[6.3]{Besse} or \cite{Kreyssig} for a coordinate--free definition).
\end{definition}

The following proposition can be found in \cite[Proposition 2.5]{CriadoDelReyDiscreteAndContinuous}. It summarizes some of the properties of the volume density.

\begin{proposition}\label{prop:prpiedades_volume_density} The volume density satisfies the following properties:
\begin{enumerate}
\item $\omega_x$ is smooth in any normal neighborhood around $x$.
\item $\omega_x(x) = 1$.
\item $\omega_x(y) > 0$ if $d(x,y)<\textup{inj}(x)$.
\item $\omega_x(y) = 0$ if and only if $y$ is conjugate to $x$.
\item $\omega_x(y) = \omega_y(x)$.
\end{enumerate}
\end{proposition}

\begin{definition}\label{def:lh} Let $\M$ be a Riemannian manifold, and let $x\in\M$ be any point. We will say that $\M$ is \emph{locally harmonic at $x$} if there exists an $\eps > 0$ such that $\omega_x$ is radially symmetric on $B(x,\eps)$. In other words, if there is a function $\Omega_x:[0,\eps)\ra\R$ such that $\omega_x(y) = \Omega_x(d(x,y))$ for every $y\in B(x,\eps)$. We will say that $\M$ is \emph{locally harmonic} if it is locally harmonic at $x$ for every $x\in\M$.
\end{definition}

See \cite[Chapter 6]{Besse} for more on locally harmonic manifolds. It is easy to see that the Euclidean space $\R^n$ is locally harmonic with $\Omega_x(r) \equiv 1$. The so--called Compact Rank One Symmetric Spaces (CROSS) are also locally harmonic manifolds. These are the sphere $\S^n$, the projective spaces $\R\mathbb{P}^m$, $\C\mathbb{P}^m$, $\mathbb{H}\mathbb{P}^m$ and the Cayley plane $\mathbb{O}\mathbb{P}^2$. We will also be using the concept of a Blaschke manifold.

\begin{definition} We will say that $\M$ is a \emph{Blaschke manifold} if $\M$ is compact and if the injectivity radius of $\M$ coincides with its diameter.
\end{definition}

As the next result shows, the Green's function on a locally harmonic Blaschke manifold is radially symmetric and thus it is easier to study. The proof can be found in \cite{CriadoDelReyDiscreteAndContinuous}.

\begin{proposition}\label{thm:green_phi} Let $\M$ be a locally harmonic Blaschke manifold of diameter $D$. Then the Green's function is given by $G(x,y) = \phi(d(x,y))$, where
\begin{equation}\label{eq:phiprima}
\phi'(r) = -\frac{V^{-1}\int_r^D v(t)dt}{v(r)}
\end{equation} and $\phi$ is a primitive of $\phi'$ making $\int_\M \phi(d(x,y)) = 0$. Here $v(r)$ denotes the volume of the geodesic sphere of radius $r$, which in the case of locally harmonic Blaschke manifolds does not depend on the center of the sphere (see \cite[Remark A6]{CriadoDelReyDiscreteAndContinuous}).
\end{proposition}

We now present a result similar to \cite[Theorem 14]{GustafssonPartialBalayage} for locally harmonic Blaschke manifolds of any dimension.

\begin{theorem}\label{thm:harmonic_geodesic} Let $\M$ be a locally harmonic Blaschke manifold of dimension $n \geq 2$ and diameter $D$. Then, for every $0 < r < D$, we have that
\[
B(p,r) = B^h(p,V(r)),
\] where $V(r) = \textup{vol}(B(p,r))$ is the volume of the geodesic ball (which does not depend on the center of the ball for locally harmonic Blaschke manifolds).
\end{theorem}

The proof of Theorem \ref{thm:harmonic_geodesic} will require two lemmas. In what follows we will adopt the following notations: for a compact manifold $\M$, $dS_r$ will denote the induced Riemannian measure on the geodesic sphere $S(p,r)$, and $v_p(r) = \int_{S(p,r)}dS_r$ will denote the volume of the sphere. We will denote the uniform probability measure on $S(p,r)$ by $\sigma_r$ (so $\sigma_r = dS_r/v_p(r)$). Finally, we will denote by $V_p(r) = \vol(B(p,r))$ the volume of the geodesic ball around $p$.

\begin{remark} If $x\in\M$ is any point, then for any $y\in \text{Cut}(x)$ we have the inequalities
\[
\inj(\M) \leq \inj(x) \leq d(x,y) \leq \diam(\M).
\] (see Remark \ref{rmk:injectivity}). If $\M$ is Blaschke, then all of the above quantities coincide, and so for every continuous function $f:\M\ra\R$,
\[
\int_\M f(x)\dvol(x) = \int_0^D \int_{x\in S(p,r)}f(x)dS_r(x),
\] where $p\in\M$ is any point, since $\text{Cut}(p)$ has measure zero.
\end{remark}

The following result is just a remark about integration over spheres on locally harmonic manifolds.

\begin{lemma}\label{lem:515}Let $\M$ be a compact manifold, and let $p\in\M$ be any point. Assume that there exist a real number $\eps$ with $0 < \eps \leq \textup{inj}(\M)$ and a function $\Omega_p:[0,\eps)\ra \R$ such that $\omega_p(x) = \Omega_p(d(p,x))$ for every $x\in B(p,\eps)$. Then, for every integrable function $f:\M\ra\R$ and for every $0 < r < \eps$, we have that
\begin{equation}\label{eq:rmk_integration}
\int_{y\in S(p,r)}f(y)dS_r(y) = \int_{\theta\in\S^{n-1}}f(\exp_p r\theta)r^{n-1}\Omega_p(r)d\theta,
\end{equation} In particular, taking $f\equiv 1$, one has
\[
v_p(r) = \textup{vol}(\S^{n-1})r^{n-1}\Omega_p(r).
\]
\end{lemma}

\begin{proof} Let us denote by $\S^n(t) = \{v\in\R^{n+1} : \|v\| = t\}$ the Euclidean sphere of dimension $n$ and radius $t$ together with its standard measure. Observe that
\begin{align*}
\int_0^r\int_{\theta\in \S^{n-1}}f(\exp_p t\theta)t^{n-1}\Omega_p(t)d\theta dt &= \int_0^r \int_{\theta\in \S^{n-1}(t)}f(\exp_p\theta)\Omega_p(t)d\theta dt\\
&= \int_0^r \int_{\theta\in \S^{n-1}(t)}f(\exp_p\theta)\sqrt{|g|}(\exp_p\theta)d\theta dt\\
&= \int_{B(p,r)}f(y)\dvol(y)\\
&= \int_0^r \int_{y\in S(p,t)}dS_r(y)dt.
\end{align*} The result follows by taking the derivative with respect to $r$.
\end{proof}

In the next lemma we provide an explicit expression for the potential of the uniform measure on a geodesic sphere in the case of locally harmonic manifolds.

\begin{lemma}\label{lem:fr} Let $\M$ be a compact manifold, and let $p\in\M$ be any point. Assume that there exist a real number $\eps$, with $0 < \eps \leq \textup{inj}(\M)$, and a function $\Omega_p:[0,\eps)\ra \R$ such that $\omega_p(x) = \Omega_p(d(p,x))$ for every $x\in B(p,\eps)$. Then, for every $0 < r < \eps$, the function
\[
f_r(x) = \int_{y\in S(p,r)}G(x,y)d\sigma_r(y),
\] is continuous and, moreover,
\[
f_r(x) = \left\{\begin{array}{ll}
G^{\delta_p}(x)+V^{-1}\int_0^r \frac{V_p(u)}{v_p(u)}du & \text{ if }x\in\M\setminus B(p,r),\\
& \\
G^{\delta_p}(z)+V^{-1}\int_0^r\frac{V_p(u)}{v_p(u)}du-\int_{d(p,x)}^r \frac{du}{v_p(u)} & \text{ if }x\in B(p,r).
\end{array}\right.
\] where $z\in S(p,r)$ is an arbitrary point.
\end{lemma}

\begin{proof} Let $t = d(p,x)$. Consider the function
\[
F(u) = \int_{S(p,u)}G(x,y)d\sigma_u(y), \qquad u\in (0,\eps),
\] and observe that $f_r(x) = F(r)$. By taking normal coordinates around $p$, if $u\neq r$, we can use \eqref{eq:rmk_integration} in Lemma \ref{lem:515} to obtain that
\begin{align*}
F(u) &= \frac{1}{\int_{\theta\in\S^{n-1}}u^{n-1}\Omega_p(u)d\theta}\int_{\theta\in \S^{n-1}}G(x,\exp_p u\theta)u^{n-1}\Omega_p(u)d\theta\\
&= \frac{1}{\vol(\S^{n-1})}\int_{\theta\in \S^{n-1}}G(x,\exp_p u\theta)d\theta.
\end{align*} Hence, if $u\neq r$, then
\[
F'(u) = \frac{1}{\vol(\S^{n-1})}\int_{\theta\in \S^{n-1}}\frac{d}{du}G(x,\exp_p u\theta)d\theta = \int_{S(p,u)}\nabla_N G(x,y)d\sigma_u(y),
\] where $N$ is the outward unit normal vector field on $S(p,u)$.

Suppose that $x\in\M\setminus \overline{B(p,r)}$, so that $r < t$. Then there exists $\delta > 0$ such that, for every $u\in (0,r+\delta)$, the point $x$ does not belong to $\overline{B(p,u)}$. By Green's Second Identity, we get
\[
v_p(u)F'(u) = -\int_{y\in B(p,u)}\Delta G(x,y)\dvol(y) = V^{-1}V_p(u), \qquad u\in (0,r+\delta).
\] Moreover, by continuity of the Green's function,
\[
\lim_{u\ra 0^+}F(u) = G^{\delta_p}(x).
\] Therefore,
\[
f_r(x) = F(r) = G^{\delta_p}(x)+V^{-1}\int_0^r \frac{V_p(u)}{v_p(u)}du.
\] To see that this formula is also valid when $r = t$, we will show that $f_r$ is continuous. For each $\delta > 0$, consider the compact set
\[
E_\delta = \{(x,y)\in\M\times\M : G(x,y)\leq \delta^{-1}\}.
\] Since $G|_{E_\delta}$ is continuous, by  Tietze's Extension Theorem there is a continuous function $G_\delta:\M\times\M\ra\R$ with $G_\delta|_{E_\delta} = G|_{E_\delta}$, and $G_\delta\leq \delta^{-1}$. Then the function
\[
f_\delta(x) = \int_{y\in S(p,r)}G_\delta(x,y)d\sigma_r(y)
\] is continuous. For each $x\in\M$, we have that $|G_\delta(x,y)|\leq |G(x,y)|$ for every $y\in\M$. Let us show that $y\mapsto G(x,y)$ is $\sigma_r$--integrable. In what follows $C$ will represent an arbitrary positive constant. If $n > 2$, then $G(x,y)\leq Cd(x,y)^{2-n}$ (see, for example, \cite[Theorem 4.13]{Aubin}), so
\[
\int_{y\in S(p,r)}|G(x,y)|d\sigma_r(y) \leq C\int_{y\in S(p,r)}d(x,y)^{2-n}d\sigma_r(y).
\] Taking normal coordinates around $p$, we can write
\[
C\int_{y\in S(p,r)}d(x,y)^{2-n}d\sigma_r(y) = C\int_{v\in \S^{n-1}(r)}d(\exp_pu,\exp_pv)^{2-n}dv,
\] where $u = \exp_p^{-1}x$. Since $\exp_p^{-1}$ is smooth on the closed ball $\overline{B(p,r)}$, it is Lipschitz continuous, and hence there is a constant $L> 0$ such that $\|u-v\|\leq Ld(\exp_pu,\exp_pv)$. Therefore,
\[
C\int_{v\in \S^{n-1}(r)}d(\exp_pu,\exp_pv)^{2-n}\leq C\int_{v\in \S^{n-1}(r)}\|u-v\|^{2-n}dv < \infty.
\] The case $n = 2$ is similar using the fact that $G(x,y)\leq C\log\|x-y\|^{-1}$. Since $y\mapsto G(x,y)$ is $\sigma_r$--integrable, the Dominated Convergence Theorem applies and we get
\[
f_r(x) = \int_{y\in S(p,r)}\lim_{\delta\ra 0}G_\delta(x,y)d\sigma_r(y) = \lim_{\delta\ra 0}\int_{y\in S(p,r)}G_\delta(x,y)d\sigma_r(y).
\] Now $f_r$ is the pointwise limit of a non--decreasing sequence of continuous functions on a compact manifold. By Dini's Theorem, the convergence is uniform and thus $f$ is continuous.

Finally, if $x\in S(p,r)$, then we can take a sequence $(x_m)_m$ of points in $\M$ with $x_m\ra x$ and $d(p,x_m) > r$ for every $m$. Then
\begin{align*}
f(x) &= \lim_{m\ra\infty}f(x_m)\\
&= \lim_{m\ra\infty}G^\delta_p(x_m)+V^{-1}\int_0^r \frac{V_p(u)}{v_p(u)}du\\
&= G^{\delta_p}(x)+V^{-1}\int_0^r\frac{V_p(u)}{v_p(u)}du,
\end{align*} since $G^{\delta_p}$ is continuous at $x$.

To finish the proof, assume now that $r > t$. Then there exists $\delta>0$ such that $x\in B(p,u)$ for every $u\in (t,r+\delta)$. By Green's Second Identity,
\[
v_p(u)F'(u) = \int_{\M\setminus B(p,u)}\Delta G(x,y)\dvol(y) = V_p(u)V^{-1}-1, \qquad u\in (t,r+\delta).
\] Then
\[
F(r) = F(t)+\int_t^r \frac{V^{-1}V_p(u)-1}{v_p(u)}du,
\] and by the continuity of $f$ and its expression for $r = t$, we obtain that if $z\in S(p,r)$ is any point, then
\begin{align*}
f(x) &= F(r) = G^{\delta_p}(z)+V^{-1}\int_0^t\frac{V_p(u)}{v_p(u)}du+\int_t^r\frac{1-V^{-1}V_p(u)}{v_p(u)}du\\
&= G^{\delta_p}(z)+V^{-1}\int_0^r\frac{V_p(u)}{v_p(u)}du-\int_t^r\frac{du}{v_p(u)}.
\end{align*}
\end{proof}

We can now prove Theorem \ref{thm:harmonic_geodesic} stating that on locally harmonic Blaschke manifolds geodesic and harmonic balls coincide as families.

\begin{proof}[Proof of Theorem \ref{thm:harmonic_geodesic}] Let us denote by $h_r$ the potential of the measure $\chi_{B(p,r)}\vol$. This is,
\[
h_r(x) = \int_{y\in B(p,r)}G(x,y)\dvol(y),
\] and let
\begin{equation}\label{eq:def_br}
b_r(x) = h_r(x)-V^{-1}\int_0^rv(r)\int_0^t\frac{V(u)}{v(u)}dudt.
\end{equation} Applying Lemma \ref{lem:fr}, if $x\in \M\setminus B(p,r)$, then
\begin{align*}
h_r(x) &= \int_0^r v(r)\int_{S(p,t)}G(x,y)d\sigma_t(y)dt\\
&= \int_0^r v(r)\left(G^{\delta_p}(x)+V^{-1}\int_0^t\frac{V(u)}{v(u)}du\right)dt\\
&= V(r)G^{\delta_p}(x)+V^{-1}\int_0^r v(r)\int_0^t\frac{V(u)}{v(u)}dudt,
\end{align*} so that
\begin{equation}\label{eq:wrg}
b_r(x) = V(r)G^{\delta_p}(x) \quad \text{ on } \M\setminus B(p,r).
\end{equation} If $x\in B(p,r)$, then again by Lemma \ref{lem:fr} we get that
\begin{equation}\label{eq:wrg2}
b_r(x) = V(r)G^{\delta_p}(z)-\int_0^r v(t)\int_{d(p,x)}^t \frac{du}{v(u)}dt,
\end{equation} for some $z\in S(p,r)$. From the expression \eqref{eq:phiprima} for $\phi'$ in Proposition \ref{thm:green_phi}, we see that the function $x\mapsto V(r)G^{\delta_p}(x)$ is decreasing with the distance to $p$, and thus $V(r)G^{\delta_p}(z) < V(r)G^{\delta_p}(x)$ for every $z\in S(p,r)$ and every $x\in B(p,r)$. This fact together with \eqref{eq:wrg2} gives us
\begin{equation}\label{eq:wrg3}
b_r(x) < V(r)G^{\delta_p}(x) \quad \text{ on } B(p,r).
\end{equation} From the definition of $b_r$ \eqref{eq:def_br} it is clear that
\[
\Delta b_r = \Delta h_r = \chi_{B(p,r)}\vol-V^{-1}V(r)\vol,
\] and in particular,
\begin{equation}\label{eq:wrg4}
\Delta b_r \leq (1-V^{-1}V(r))\vol \quad \text{ on } \M.
\end{equation} Moreover, Lemma \ref{lem:fr} implies that $h_r$ is continuous, and so $b_r$ is also continuous. Now \eqref{eq:wrg}, \eqref{eq:wrg3} and \eqref{eq:wrg4} together with Lemma \ref{lem:lemita} imply that $\M\setminus B(p,r)\subseteq \M\setminus B^h(p,V(r))$. That is,
\begin{equation}\label{eq:contenido1}
B^h(p,V(r)) \subseteq B(p,r).
\end{equation}

Now let $u_a$ be the function from Definition \ref{def:harmonic_ball} with $a = V(r)$. We have the following:

\begin{enumerate}
\item $u_a \leq 1-V(r)V^{-1}$ everywhere on $\M$ and, in particular, in $B(p,r)$.
\item From the definition, $\Delta b_r = \Delta h_r = 1-V(r)V^{-1}$ on $B(p,r)$.
\item From \eqref{eq:wrg}, $b_r = V(r)G^{\delta_p}$ on $\M\setminus B(p,r)$.
\item From \eqref{eq:contenido1} and the definition of $B^h(p,V(r))$, $u_a = V(r)G^{\delta_p}$ on $\M\setminus B(p,r)$.
\end{enumerate} By the Maximum Principle, $u_a \leq b_r$ everywhere on $\M$. However, since $b_r\leq V(r)G^{\delta_p}$ and $\Delta b_r \leq 1-V(r)V^{-1}$ on $\M$, $b_r$ is a competing function for the obstacle problem \eqref{eq:obstacle_ua} for $u_a$. Hence we get $b_r = u_a$. Finally, the strict inequality in \eqref{eq:wrg3} together with \eqref{eq:wrg} imply that
\[
B^h(p,V(r)) = \{x\in\M : u_a < aG^{\delta_p}\} = \{x\in \M : b_r < aG^{\delta_p}\} = B(p,r).
\]
\end{proof}

As a consequence of Theorem \ref{thm:main1} and Theorem \ref{thm:harmonic_geodesic}, we obtain a separation distance result for locally harmonic Blaschke manifolds.

\begin{corollary}\label{cor:dsep_lh} Let $\M$ be a locally harmonic Blaschke manifold of dimension $n \geq 2$ and let $\omega_N^* = \{x_1^*, ..., x_N^*\}$ be a collection of $N\geq 2$ points minimizing the discrete Green energy. Then,
\[
x_i^*\notin B(x_j^*,r_N),
\] where $r_N$ is the radius such that $V(r_N) = V/(N-1)$. In particular, since $V(r) \leq Cr^n$ for some constant $C>0$, there exists a constant $c>0$ such that, for every sequence of minimizers $(\omega_N^*)_N$,
\[
\textup{dsep}(\omega_N^*) \geq cN^{-1/n}.
\]
\end{corollary}

We can now write down explicit bounds for the separation distance in the case of the CROSS. The proof of Theorem \ref{thm:main2} consists of computing the radius $r_N$ in Corollary \ref{cor:dsep_lh}.

\begin{proof}[Proof of Theorem \ref{thm:main2}] The CROSS are known to be locally harmonic manifolds. For the sphere $\S^n$, we have that
\[
r^{n-1}\Omega(r) = \sin^{n-1}r.
\] By Lemma \ref{lem:515}, the volume of the ball $V(r) = V_{\mathbb{S}^n}(r)$ is
\begin{align*}
V(r) &= \int_0^r \int_{\theta\in S(p,t)}d\theta dt\\
&= \int_0^r \vol(\S^{n-1})\sin^{n-1}tdt
\end{align*} From the fact that $\sin t \leq t$, we obtain that
\[
V(r) \leq \frac{\vol(\S^{n-1})}{n}r^n.
\] Hence,
\[
V\left[\left(\frac{n\,\vol(\S^n)}{\vol(\S^{n-1})(N-1)}\right)^{1/n}\right]\leq \frac{\vol(\S^n)}{N-1},
\] Since $V(r)$ is an increasing function, the radius $r_N$ from Corollary \ref{cor:dsep_lh} satisfies
\begin{equation}\label{eq:rn_sn}
r_N \geq \left(\frac{n\,\vol(\S^n)}{\vol(\S^{n-1})}\right)^{1/n}(N-1)^{-1/n}.
\end{equation} Now, the volume of the sphere $\S^n$ is given by
\begin{equation}\label{eq:vol_sn}
\vol(\S^n) = \frac{2\pi^{\frac{n+1}{2}}}{\Gamma\left(\frac{n+1}{2}\right)} = \left\{\begin{array}{ll}
\frac{2^{\frac{n+3}{2}}\pi^{\frac{n+1}{2}}}{n!} & \text{ if }n\text{ is even,}\\
& \\
\frac{2\pi^{\frac{n+1}{2}}}{\left(\frac{n+1}{2}-1\right)!} & \text{ if }n\text{ is odd,}
\end{array}\right.
\end{equation} and thus the value of the constant in \eqref{eq:rn_sn} is
\[
C_{\S^n} = \left(\frac{n\,\vol(\S^n)}{\vol(\S^{n-1})}\right)^{1/n} = \left(\frac{n\,2\pi^{\frac{n+1}{2}}\Gamma\left(\frac{n}{2}\right)}{2\pi^{\frac{n}{2}}\Gamma\left(\frac{n+1}{2}\right)}\right)^{1/n} = \left(\frac{n\sqrt{\pi}\,\Gamma\left(\frac{n}{2}\right)}{\Gamma\left(\frac{n+1}{2}\right)}\right)^{1/n}.
\]

The projective spaces $\mathbb{K}\mathbb{P}^m$ (where $\mathbb{K}\in\{\R,\C,\mathbb{H},\mathbb{O}\}$, and we assume that $m = 2$ if $\mathbb{K} = \mathbb{O}$) have dimension $n = m\,\dim_\R\mathbb{K}$. Their volume densities were calculated in \cite[Proposition 3.3.1]{Kreyssig} under the assumption that they all have diameter $\pi$. Here we will adopt the convention that the diameter of the projective spaces is $\pi/2$. To compute the corresponding volume densities, we can look at the last statement in Lemma \ref{lem:515}, which tells us how does volume density change if we scale distances on the manifold. It is easy to see that the volumes $v(r)$ and $\hat{v}(r)$ of geodesic spheres in the diameter $\pi/2$ setting and in the diameter $\pi$ setting, respectively, are related by
\[
v(r) =2^{1-n}\hat{v}(2r).
\] Therefore, using the formulas in \cite[Proposition 3.3.1]{Kreyssig},
\begin{align*}
r^{n-1}\Omega(r) &= 2^{1-n}\,2^{\frac{\dim_\R\Kk}{2}(m-1)}(\sin 2r)^{\dim_\R\Kk-1}(1-\cos 2r)^{\frac{\dim_\R\Kk}{2}(m-1)}\\
&= 2^{1-n}\,2^{\frac{n}{2}-\frac{\dim_\R\Kk}{2}}\,2^{\dim_\R\Kk-1}\sin^{\dim_\R\Kk-1}r\,\cos^{\dim_\R\Kk-1}r\,2^{\frac{n}{2}-\frac{\dim_\R\Kk}{2}}\sin^{n-\dim_\R\Kk}r\\
&= \sin^{n-1}r\,\cos^{\dim_\R\Kk-1}r.
\end{align*} Now we can apply the same argument as in the case of the sphere. Since $\sin t\leq t$ and $\cos t\leq 1$, we get
\begin{equation}\label{rn_kn}
r_N\geq \left(\frac{n\,\vol(\Kk\mathbb{P}^m)}{\vol(\S^{n-1})}\right)^{1/n}(N-1)^{-1/n}.
\end{equation} Since projective spaces (except for the Cayley plane) are quotients of spheres, their volumes can be easily calculated (see, for example, \cite{VolumesManifolds}). The volumes are the following:

\begin{align}\label{eq:volumes_projecive}
\nonumber \vol(\R\mathbb{P}^m) &= \frac{\pi^{\frac{m+1}{2}}}{\Gamma\left(\frac{m+1}{2}\right)}, & \vol(\C\mathbb{P}^m) &= \frac{\pi^m}{m!},\\
\vol(\mathbb{H}\mathbb{P}^m) &= \frac{\pi^{2m}}{(2m+1)!}, &  \vol(\mathbb{O}\mathbb{P}^2) &=  \frac{3!\pi^8}{11!}.
\end{align} By using the formulas in \eqref{eq:vol_sn} and \eqref{eq:volumes_projecive}, we can start computing the values for the constant in the right hand side of \eqref{eq:rn_sn}. For the real projective space $(n = m)$, we get that
\[
C_{\R\mathbb{P}^m} = \left(\frac{n\,\vol(\mathbb{R}\mathbb{P}^m)}{\vol(\S^{n-1})}\right)^{1/n}
= \left(\frac{m\,\pi^\frac{m+1}{2}\Gamma\left(\frac{m}{2}\right)}{\Gamma\left(\frac{m+1}{2}\right)\,2\pi^{\frac{m}{2}}}\right)^{1/m} = \left(\frac{m\sqrt{\pi}\,\Gamma\left(\frac{m}{2}\right)}{2\,\Gamma\left(\frac{m+1}{2}\right)}\right)^{1/m}.
\] For the rest of the projective spaces, we use the fact  that $\S^{n-1}$ is always an odd dimensional sphere and look at \eqref{eq:vol_sn}. In the case of $\C\mathbb{P}^m$ ($n = 2m$),
\[
C_{\C\mathbb{P}^m} = \left(\frac{n\,\vol(\C\mathbb{P}^m)}{\vol(\S^{n-1})}\right)^{1/n} = \left(\frac{2m\,\pi^m(m-1)!}{m!\,2\pi^m}\right)^{1/2m} = 1.
\] In the case of $\mathbb{H}\mathbb{P}^m$ ($n = 4m$),
\[
C_{\mathbb{H}\mathbb{P}^m} = \left(\frac{n\,\vol(\mathbb{H}\mathbb{P}^m)}{\vol(\S^{n-1})}\right)^{1/n} = \left(\frac{4m\,\pi^{2m}(2m-1)!}{(2m+1)!\,2\pi^{2m}}\right)^{1/4m} = \left(\frac{1}{2m+1}\right)^{1/4m}.
\] Finally, for $\mathbb{O}\mathbb{P}^2$ ($n = 16$),
\[
C_{\mathbb{O}\mathbb{P}^2} = \left(\frac{16\,\vol(\mathbb{O}\mathbb{P}^2)}{\vol(\S^{15})}\right)^{1/16} = \left(\frac{16\cdot 3!\,\pi^8 7!}{11!\,2\pi^8}\right)^{1/16} = \left(\frac{1}{165}\right)^{1/16}.
\]
\end{proof}

\section{Separation distance in the general case}\label{sec:general}

In this last section we include the proof of Theorem \ref{thm:main3}, which essentially amounts to show that every (sufficiently small) harmonic ball contains a geodesic ball of suitable radius. In the Fermi lectures in 1998, Caffarelli stated a mean value theorem for general divergence form elliptic operators $L$. On page 9 of \cite{CaffarelliObstacleFermi} his argument requires a key test function which can be obtained by solving the problem
\begin{equation}\label{eq:Vr}
-Lw = \delta_p-\frac{1}{R^n}\chi_{\{w>0\}}
\end{equation} for $R>0$. It turns out that proving the existence of a solution to \eqref{eq:Vr} is not trivial, and in \cite{BlankDivergenceForm} Blank and Hao give a detailed proof of the statement by Caffarelli. Moreover, they show that the non--contact set $\mathcal{N}_R$ for the solution is always nested between two (Euclidean) balls, in the sense that there exist positive constants $c$ and $C$ such that
\[
\mathbb{B}_{cR}\subseteq \mathcal{N}_R \subseteq \mathbb{B}_{CR}
\] for every $R>0$. In \cite{BlankManifolds}, and more recently in \cite{BlankArmstrong}, the authors generalize some of the results in \cite{BlankDivergenceForm} to the setting of Riemannian manifolds. In this section we will simply adapt some of the arguments in \cite{BlankDivergenceForm} and \cite{BlankManifolds} to our case in order to show that every sufficiently small harmonic ball contains a geodesic ball of proportional radius. Some of the proofs are virtually the same and we will not include them.

We begin with a non--degeneracy theorem that describes how the graphs of $u_a$ and $aG^{\delta_p}$ separate one from each other near the boundary $\partial B^h(p,a)$. The proof relies on the estimate from \cite[Theorem 2.3]{BlankManifolds} which bounds the Laplacian of the squared Riemannian distance function in terms of bounds for the Ricci curvature.

\begin{theorem}[Local nondegeneracy]\label{thm:local_nd} There exist a radius $r_{ND}>0$ and a constant $C_{ND}$ such that for every $p\in\M$, $0 < a < V$, $x_0\in \overline{B^h(p,a)}$ and $0 < r \leq r_{ND}$,
\begin{equation}\label{eq:grafos}
\sup_{x\in B(x_0,r)}(aG^{\delta_p}(x)-u_a(x)) \geq C_{ND}r^2.
\end{equation}
\end{theorem}

We do not include the proof of this theorem, which simply amounts to mimic that of \cite[Theorem 2.3]{BlankManifolds}. Next, we prove a lemma that will be necessary to adapt the argument in the proof of \cite[Lemma 4.1]{BlankManifolds}.

\begin{lemma}\label{lem:cota_inf_ua} The function $u_a$ is bounded from below by $\min_{x\in\M}aG^{\delta_p}(x)$.
\end{lemma}

\begin{proof} Let $m_a = \min_{x\in\M\setminus B^h(p,a)}aG^{\delta_p}(x)$. Then $u_a \geq m_a$ on $\M\setminus B^h(p,a)$, because $u_a$ coincides with $aG^{\delta_p}$ in this set. In particular, $u_a \geq m_a$ on $\partial B^h(p,a)$, and since $u_a$ is superharmonic in $B^h(p,a)$, it follows from the Maximum Principle that $u_a \geq m_a$ also on $B^h(p,a)$. Hence $u_a \geq m_a$ everywhere on $\M$.

Now, since $aG^{\delta_p}\geq u_a$ on $B^h(p,a)$, we have that
\[
\inf_{x\in B^h(p,a)}aG^{\delta_p}(x) \geq \inf_{x\in B^h(p,a)}u_a(x) \geq m_a = \min_{x\in \M\setminus B^h(p,a)}aG^{\delta_p}(x).
\] Therefore $m_a = \min_{x\in\M}aG^{\delta_p}(x)$ and the lemma follows.
\end{proof}

%
%\begin{theorem}[Global nondegeneracy]\label{thm:global_nd} There exists a constant $C_{GN}>0$ such that for every $p\in\M$, $0 < a < V$, $x_0\in \overline{B^h(p,a)}$ and $r \geq r_{ND}$,
%\[
%\sup_{x\in B(x_0,r)}(aG^{\delta_p}(x)-u_a(x)) \geq C_{GN}r.
%\]
%\end{theorem}

Throughout the rest of this section we will fix a radius $0 < R < \inj(\M)$. The next proposition shows that for every small enough $a$, the harmonic ball $B^h(p,a)$ is positive distance away from $\partial B(p,R)$.

\begin{proposition}\label{thm:positive_distance} There exists a constant $a_0$ with $0 < a_0 < V$ such that, for every $p\in\M$ and for every $0 < a \leq a_0$, $\overline{B^h(p,a)}\subset B(p,R)$.
\end{proposition}

\begin{proof} Let $\delta$ be a number with $0 < \delta < \min\{r_{ND},R/2\}$, where $r_{ND}$ is the radius from Theorem \ref{thm:local_nd}. Let us denote by $\eta_{2\delta}(\partial B(p,R))$ the $2\delta$--tubular neighborhood of $\partial B(p,R)$. Since $p\notin \eta_{2\delta}(\partial B(p,R))$, there is a constant $C$ such that $G^{\delta_p}\leq C$ on $\overline{\eta_{2\delta}(\partial B(p,R))}\cap\overline{B(p,R)}$. Let $y_0\in B(p,R)$ be a point such that $d(y_0,\partial B(p,R)) = \delta$, and assume by contradiction that $y_0\in \overline{B^h(p,a)}$ for every $0 < a < V$. Then, by \eqref{eq:grafos} in Theorem \ref{thm:local_nd},
\[
aC-\inf_{y\in B(y_0,\delta)}u_a(y) \geq \sup_{y\in B(y_0,\delta)}(aG^{\delta_p}(y)-u_a(y)) \geq C_{ND}\delta^2,
\] which implies
\begin{equation}\label{eq:contr1}
\inf_{y\in B(y_0,\delta)}u_a(y) \leq aC-C_{ND}\delta^2.
\end{equation} Let $m = \min_{x\in\M}G^{\delta_p}$. Then $am = \min_{x\in\M}aG^{\delta_p}$ and, by Lemma \ref{lem:cota_inf_ua}, $am \leq u_a$ everywhere on $\M$. Hence \eqref{eq:contr1} implies
\[
am \leq aC - C_{ND}\delta^2,
\] which gives a contradiction as soon as $a$ is sufficiently small.

We have shown that, as soon as $a$ is bounded above by some $a_0$, $B^h(p,a)\cap B(p,R)$ is either empty or positive distance away from $\partial B(p,R)$. The fact that $B^h(p,a)$ is connected (Proposition \ref{prop:bharm_connected}) and $p\in B^h(p,a)\cap B(p,R)$ finishes the proof.
\end{proof}

With Proposition \ref{thm:positive_distance} in hand we can now rigorously prove that the mean value sets $D_p(r)$ from \cite{BlankManifolds} and harmonic balls $B^h(p,a)$ from \cite{GustafssonPartialBalayage} are the same subsets of $\M$ as long as $r$ and $a$ are sufficiently small. The mean value sets are defined in terms of $G_R^{\delta_p}$, the Green's function of the geodesic ball $B(p,R)$. This function satisfies:
\begin{enumerate}
\item $G_R^{\delta_p}(x) > 0$ on $B(p,R)\setminus\{p\}$,
\item $G_R^{\delta_p}(x)$ is smooth on $B(p,R)\setminus\{p\}$,
\item $G_R^{\delta_p}(x) = 0$ for every $x\in \partial B(p,R)$,
\item $\lim_{x\ra p}G_R^{\delta_p}(x) = +\infty$,
\item $\Delta G_R^{\delta_p} = \delta_p$ in the sense of distributions.
\end{enumerate} (See \cite[Ch. 4]{Aubin}). If $\Omega\subset\M$ is an open set with smooth boundary, we denote by $W^{1,2}_0(\Omega)$ the completion of $C^\infty_0(\Omega)$ with respect to the Sobolev norm in $W^{1,2}(\Omega)$. The definition of the mean value sets is as follows:

\begin{definition}\label{def:mean_value_sets} Let $p\in\M$ be any point and let $r>0$. There is a unique function $w_r$ which minimizes the functional
\[
\hat{J}_{r,R}(w) = \int_{B(p,R)}\|\nabla w\|^2-2r^{-n}w
\] among all functions in $W^{1,2}_0(B(p,R))$ with $w \leq G_R^{\delta_p}$. The function $w_r$ is the solution to the obstacle problem
\[
\left\{\begin{array}{ll}
\Delta w = \frac{1}{r^{n}}\chi_{\{w < G_R^{\delta_p}\}} & \text{ in }B(p,R),\\
w = 0 & \text{ on }\partial B(p,R).
\end{array}\right.
\] $w_r$ is continuous and we define the mean value set $D_p(r)$ by
\[
D_p(r) = \{w_r < G_R^{\delta_p}\}.
\]
\end{definition}

\begin{proposition}\label{thm:bh_mvs} There is an $r_0>0$ such that, for every $p\in\M$ and for every $0 < r \leq r_0$, the sets $D_p(r)$ and $B^h(p,r^n)$ are compactly contained in $B(p,R)$, and, moreover,
\[
D_p(r) = B^h(p,r^n).
\]
\end{proposition}

\begin{proof} In the proof of Lemma 4.1 in \cite{BlankManifolds} it is shown that there exists a radius $\tilde{r}_0 > 0$ such that, for every $p\in\M$ and for every $0 < r \leq \tilde{r}_0$, $\overline{D_p(r)}\subset B(p,R)$. Let $a_0$ be the constant from Theorem \ref{thm:positive_distance}, and let $r_0 = \min\{a_0^{1/n},\tilde{r}_0\}$.

Let $0 < r \leq r_0$ be any radius, and let us set $a = r^n$. From Theorem 7 in \cite{GustafssonPartialBalayage}, the function $u_a$ is the unique minimizer of the functional
\[
J_a(u) = \int_\M \|\nabla u\|^2-2\t_a u
\] among all the functions $u\in W^{1,2}(\M)$ with $u\leq aG^{\delta_p}$. By Proposition \ref{thm:positive_distance}, if we let $\tilde{u}_a$ be the unique minimizer of the functional
\[
J_{a,R}(u) = \int_{B(p,R)}\|\nabla u\|^2-2\t_a u
\] in the set
\[
K_{a,R} = \{u\in W^{1,2}(B(p,R)) : u-aG^{\delta_p}\in W^{1,2}_0(B(p,R)),\,u\leq aG^{\delta_p}\},
\] then $u_a$ is equal to $\tilde{u}_a$ extended to all $\M$ by $\tilde{u}_a = aG^{\delta_p}$ on $\M\setminus B(p,R)$. Consider now the set $\hat{K}_{a,R}$ given by
\[
\hat{K}_{a,R} = \{w\in W^{1,2}_0(B(p,R)) : w \leq G_R^{\delta_p}\}.
\] Let $s_R$ be the solution to the problem
\[
\left\{\begin{array}{ll}
\Delta s = -V^{-1} & \text{ in }B(p,R),\\
s = G^{\delta_p} & \text{ on }\partial B(p,R).
\end{array}\right.
\] It is easy to check that $\vphi(u) = \frac{1}{a}u-s_R$ is a bijective map between $K_{a,R}$ and $\hat{K}_{a,R}$ with inverse $\vphi^{-1}(w) = a(w+s_R)$ (just note that $G_R^{\delta_p} = G^{\delta_p}-s_R$ in $\overline{B(p,R)}$). Hence, the function $\vphi(u)$ is the unique minimizer to the functional $J_{a,R}\circ \vphi^{-1}$. Let us compute the expression for this functional. If $w\in \hat{K}_{a,R}$, then
\begin{align*}
J_{a,R}\circ\vphi^{-1}(w) &= \int_{B(p,R)}\|\nabla a(w+s_R)\|^2-2(1-aV^{-1})a(w+s_R)\\
&= \int_{B(p,R)}a^2\|\nabla w\|^2+a^2\|\nabla s_R\|^2+2a\langle \nabla w,\nabla s_R\rangle\\
&\quad -2aw-2as_R+2a^2V^{-1}w+2a^2V^{-1}s_R
\end{align*} Since $w\in W^{1,2}_0(B(p,R))$, by Green's First Identity we get that
\[
\int_{B(p,R)}\langle \nabla w,\nabla s_R\rangle = \int_{B(p,R)}w\Delta s_R = \int_{B(p,R)}-V^{-1}w.
\] Therefore, we can write
\begin{align*}
J_{a,R}\circ\vphi^{-1}(w) &= \frac{1}{a^2}\left(\int_{B(p,R)}\|\nabla w\|^2-2a^{-1}w\right)+\int_{B(p,R)}a^2\|\nabla s_R\|^2-2as_R+2aV^{-1}s_R.
\end{align*} It follows that $w$ minimizes $J_{a,R}\circ \vphi^{-1}$ if and only if it minimizes the functional $\hat{J}_{R,r}$ from Definition \ref{def:mean_value_sets} for $r = a^{1/n}$. By uniqueness of the minimizers, $\tilde{u}_a = w_{a^{1/n}}$ (or $\tilde{u}_{r^n} = w_r$) and
\[
B^h(p,a) = \{\tilde{u}_a < G^{\delta_p}\} = \{w_r < G_R^{\delta_p}\} = D_p(r).
\]
\end{proof}

Proposition \ref{thm:positive_distance} allows us to work in coordinates and use results for elliptic operators defined on subsets of $\R^n$. We consider operators of the form
\[
L = D_i a^{ij}D_j,
\] with smooth coefficients $a^{ij}$ which are uniformly elliptic in the sense that there exist ellipticity constants $\Lambda\geq \lambda > 0$ such that
\[
\lambda\|\xi\|^2 \leq a^{ij}\xi_i\xi_j \leq \Lambda \|\xi\|^2
\] for every vector $\xi = (\xi_1, ..., \xi_n)\in\R^n$. In this setting, the Riemannian Laplacian has the coordinate expression
\[
\Delta = -\frac{1}{\sqrt{|g|}}L,
\] where $a^{ij} = \sqrt{|g|}g^{ij}$.

\begin{remark}\label{rmk:ellipticity_uniform} Since $\M$ is compact and $R < \inj(\M)$, we can look at the properties of the volume density in Proposition \ref{prop:prpiedades_volume_density} to see that $\omega_x(y)$ is bounded above by $1$ and below by a positive constant in the set
\[
\{(x,y)\in \M\times\M : d(x,y)\leq R\}.
\] The coordinate expression for $\omega_x(y)$ in normal coordinates is precisely $\sqrt{|g|}$ (see Definition \ref{def:volume_density}), and hence this function will share the same bounds for every $p\in\M$. Moreover, since the eigenvalues of $g^{-1}$ can be bounded in terms of $\sqrt{|g|}$, we can always take the same ellipticity constants $\Lambda \geq \lambda > 0$, not depending on $p$, for the expression of the Laplacian in normal coordinates on $B(p,R)$.
\end{remark}

In what follows we will use the same notation for functions in $\M$ and for their coordinate expressions in $\R^n$ obtained by composing with the exponential map $\exp_p$. The following result is analogous to \cite[Lemma 4.4]{BlankDivergenceForm}.

\begin{lemma}\label{lem:cota_k} Let $L$ be a differential operator defined on the Euclidean ball $\mathbb{B}_R = \{x\in\R^n : \|x\| < R\}$ of the form
\[
L = D_i a^{ij}D_j,
\] where the $a^{ij}$ are smooth and uniformly elliptic with ellipticity constants $\lambda$ and $\Lambda$. Let $b$ be a function with $b \leq 1$. Assume that there is a continuous function $w$ that satisfies
\[
\left\{\begin{array}{ll}
Lw = -\frac{b}{r^n} & \text{ in }\mathbb{B}_R,\\
w = 0 & \text{ on }\partial \mathbb{B}_R,
\end{array}
\right.
\] where $r > 0$. Then there exists a constant $k = k(r,\lambda)>0$ such that $w \leq k$.
\end{lemma}

\begin{proof} Let $\tilde{w}$ be the solution to the problem
\[
\left\{\begin{array}{ll}
Lw = -\frac{1}{r^n} & \text{ in }\mathbb{B}_R,\\
w = 0 & \text{ on }\partial \mathbb{B}_R.
\end{array}
\right.
\] Since $b\leq 1$, we have that $-\frac{b}{r^n} \geq -\frac{1}{r^n}$. By the Maximum Principle, $w\leq \tilde{w}$. Now consider the scalar multiple of the Euclidean Laplacian given by
\[
L_\lambda = \sum_{i=1}^nD_i\lambda D_i.
\] The function
\[
\ell(x) = \frac{1}{\lambda r^n}\left(\frac{R^2-\|x\|^2}{2n}\right) \leq \frac{R^2}{2\lambda r^n n}
\] satisfies
\[
\left\{\begin{array}{ll}
L_\lambda \ell = -\frac{1}{r^n} & \text{ in }\mathbb{B}_R,\\
\ell = 0 & \text{ on }\partial \mathbb{B}_R.
\end{array}
\right.
\] Since $L_\lambda$ also has ellipticity constant $\lambda$, by \cite[Corollary 7.1]{LSW} there is a constant $K>0$ depending only on $\lambda$ such that $\tilde{w}\leq K\ell$. Hence,
\[
w \leq \tilde{w} \leq K\ell \leq \frac{KR^2}{2\lambda r^nn} =: k(r,\lambda).
\]
\end{proof}

We can now prove that every sufficiently small harmonic ball contains a geodesic ball with proportional radius.

\begin{theorem}\label{thm:bolita} There exists a constant $c>0$ such that for every $0 < a \leq a_0$, where $a_0$ is the constant from Proposition \ref{thm:positive_distance}, and for every $p\in\M$,
\[
B(p,ca^{1/n}) \subseteq B^h(p,a).
\]
\end{theorem}

\begin{proof} By Proposition \ref{thm:bh_mvs}, it suffices to prove that there is a constant $c>0$ such that, for every $0 < r \leq r_0$ and for every $p\in\M$, $B(p,cr)\subseteq D_p(r)$. Let $w_{r_0}$ be the function associated to $D_p(r_0)$ (so that $u_{r^n_0} = r_0^n(w_{r_0}+s_R)$ as in the proof of Proposition \ref{thm:harmonic_geodesic}). The function $w_{r_0}$ is continuous and satisfies
\[
\left\{\begin{array}{ll}
Lw_{r_0} = -\frac{\sqrt{|g|}}{r_0^n}\chi_{\{w_{r_0} < G_R^{\delta_p}\}} & \text{ in }\mathbb{B}_R,\\
w_{r_0} = 0 & \text{ on }\partial \mathbb{B}_R,
\end{array}
\right.
\] for the operator
\[
L = D_i \sqrt{|g|}g^{ij}D_j,
\] which is uniformly elliptic on $\mathbb{B}_R$ and, from Remark \ref{rmk:ellipticity_uniform}, has ellipticity constants $0 < \lambda \leq \Lambda$ not depending on $p$. Since $\sqrt{|g|}\leq 1$, we know from Lemma \ref{lem:cota_k} that there exists a constant $k > 0$, depending only on $r_0$ and $\lambda$, such that $w_{r_0}\leq k$. From \cite[Theorem 7.1]{LSW}, there is also a constant $\tilde{k} = \tilde{k}(\lambda) > 0 $ such that
\[
G_R^{\delta_p}(x) \geq \frac{\tilde{k}}{d(p,x)^{n-2}}
\] if $n > 2$, and
\[
G_R^{\delta_p}\geq -\tilde{k}\log d(x,y)
\] if $n = 2$. We will prove the Theorem for $n>2$. The case $n = 2$ is similar. With these bounds, we have that
\[
G_R^{\delta_p}(x) - w_r(x) \geq \frac{\tilde{k}}{d(p,x)^{n-2}}-k,
\] which is positive as long as $d(p,x) < \tilde{c} := (\tilde{k}/k)^{\frac{1}{n-2}}$. This shows that there is a constant $\tilde{c} = \tilde{c}(r_0,\lambda) > 0$ such that
\[
B(p,\tilde{c}) \subseteq \{w_{r_0} < G_R^{\delta_p}\} = D_p(r_0) = B^h(p,r_0^n).
\]

Now let $0 < s \leq 1$ be some scaling factor, and consider the map $\vphi_s:B(p,R)\ra B(p,sR)$ given by $\vphi_s = \exp_p s \exp_p^{-1}$. Let $L_s$ be the operator
\[
L_s = D_i (\sqrt{|g|}g^{ij}\circ\vphi_s)D_j.
\] It is easy to see that the function $w_{sr_0}\circ \vphi_s$ satisfies
\[
L_s w_{sr_0} = -\frac{\sqrt{|g|}\circ\vphi_s s^{2-n}}{r_0^n}\chi_{\{w_{sr_0}\circ \vphi_s < G_R^{\delta_p}\circ\vphi_s\}}.
\] Moreover,
\[
L_s G_R^{\delta_p}\circ \vphi_s = s^{2-n}\delta_p = \sqrt{|g|}\circ\vphi_s s^{2-n}\delta_p.
\] Hence, if we let $v_{sr_0} = G_R^{\delta_p}-w_{sr_0}$, then we have
\[
L_s(s^{n-2}v_{sr_0}\circ\vphi_s) = \sqrt{|g|}\circ\vphi_s\left(\delta_p+\frac{1}{r_0^n}\chi_{\{v_{sr_0}\circ\vphi_s > 0\}}\right) = \sqrt{|g|}\circ\vphi_s\left(\delta_p+\frac{1}{r_0^n}\chi_{\{s^{n-2}v_{sr_00}\circ\vphi_s > 0\}}\right).
\] Since $\sqrt{|g|}\circ\vphi_s\leq 1$ and $L_s$ has the same ellipticity constant $\lambda$ as $L$, we obtain that
\[
B(p,\tilde{c}) \subseteq \{v_{sr_0}\circ\vphi_s>0\},
\] or
\[
B(p,s\tilde{c}) \subseteq \{v_{sr_0}>0\} = D_p(sr_0).
\] Finally, the change of variables $s = \frac{r}{r_0}$ and $\tilde{c} = r_0c$ yields
\[
B(p,cr) \subseteq D_p(r)
\] for every $0 < r \leq r_0$.
\end{proof}

Theorem \ref{thm:main3} is now a trivial consequence of the above.

\begin{proof}[Proof of Theorem \ref{thm:main3}] From Theorem \ref{thm:main1}, every two distinct points $x$ and $y$ in a minimizing configuration for the Green energy satisfy
\[
x\notin B^h\left(y,\frac{V}{N-1}\right).
\] From Theorem \ref{thm:bolita}, there is a constant $c$, depending only on $\M$, such that for every sufficiently large $N$,
\[
B\left(y,c\left(\frac{V}{N-1}\right)^{1/n}\right) \subseteq B^h\left(y,\frac{V}{N-1}\right).
\] This implies that for sufficiently large $N$,
\[
d(x,y) \geq c\left(\frac{V}{N-1}\right)^{1/n}.
\] By making $c$ smaller if necessary, there is another constant $c'$ such that $d(x,y) \geq c'N^{-1/n}$ for every $N\geq 2$.
\end{proof}

\section*{Acknowledgements}  The author would like to thank Ivan Blank for his helpful comments.

\bibliographystyle{plain}

\end{document}